\documentclass[varwidth]{article}
\usepackage{tikz}
\usetikzlibrary{decorations.pathmorphing}
\usetikzlibrary{calc}

\usepackage{times}
\usepackage{graphicx}
\usepackage{amsmath}
\usepackage{amsfonts}
\usepackage{amssymb}
\usepackage[section]{placeins}

\newtheorem{theorem}{Theorem}

\newtheorem{definition}{Definition}
\newtheorem{proposition}{Proposition}
\newtheorem{remark}{Remark}

\newenvironment{proof}{{\em Proof.}}{\hfill $\diamond$ \vskip 4.5pt}

\makeatletter
\newfont{\footsc}{cmcsc10 at 8truept}
\newfont{\footbf}{cmbx10 at 8truept}
\newfont{\footrm}{cmr10 at 10truept}
\makeatother
\title{Variations on Narrow Dots-and-Boxes and Dots-and-Triangles}

\author{Adam Jobson\\
\small Department of Mathematics\\[-0.8ex]
\small University of Louisville\\[-0.8ex]
\small Louisville, KY 40292 USA\\[-0.8ex]
\small \texttt{asjobs01@louisville.edu}
\and
Levi Sledd\\
\small Department of Mathematics\\[-0.8ex]
\small Centre College\\[-0.8ex]
\small Danville, KY 40422 USA\\[-0.8ex]
\small \texttt{levi.sledd@centre.edu }
\and
Susan C. White\\
\small Department of Mathematics\\[-0.8ex]
\small Bellarmine University\\[-0.8ex]
\small Louisville, KY 40205 USA\\[-0.8ex]
\small \texttt{scwhite@bellarmine.edu}
\and
D. Jacob Wildstrom\\
\small Department of Mathematics \\[-0.8ex]
\small University of Louisville \\ [-0.8ex]
\small Louisville, KY 40292 USA\\[-0.8ex]
\small \texttt{djwild01@louisville.edu }
}
\date{\today} 

\begin{document}
\newcommand{\DBlowdot}[1]{(#1,0)}
\newcommand{\DBhighdot}[1]{(#1,1)}
\newcommand{\DBcoin}[1]{({#1+0.5},0.5)}
\newcommand{\DBlowground}[1]{({#1+0.5},0)}
\newcommand{\DBhighground}[1]{({#1+0.5},1)}
\newcommand{\DBleftground}[1]{(#1,0.5)}
\newcommand{\DBrightground}[1]{({#1+1},0.5)}
\newcommand{\DTlowdot}[1]{(#1,0)}
\newcommand{\DThighdot}[1]{({#1+0.5},{sqrt(3)/2})}
\newcommand{\DTlowcoin}[1]{({#1+0.5},{sqrt(3)/6})}
\newcommand{\DThighcoin}[1]{({#1+1},{sqrt(3)/3})}
\newcommand{\DTlowground}[1]{({#1+0.5},{sqrt(3)/6-0.5})}
\newcommand{\DThighground}[1]{({#1+1},{sqrt(3)/3+0.5})}
\newcommand{\DTleftground}[1]{({#1},{sqrt(3)/3})}
\newcommand{\DTrightground}[1]{({#1+1},{sqrt(3)/3})}
\def\nodesize{0.1}

\tikzset{dot/.style={fill=black}}
\tikzset{coin/.style={fill=white}}
\tikzset{markedcoin/.style={fill=gray}}
\tikzset{basegraph/.style={very thick}}
\tikzset{firstchain/.style={red,very thick,densely dashed}}
\tikzset{secondchain/.style={green,very thick,densely dotted}}
\tikzset{thirdchain/.style={blue,very thick,densely dashdotted}}
\tikzset{chainother/.style={thick}}

\maketitle

\begin{abstract}

We verify a conjecture of Nowakowski and Ottaway that closed $1\times n$ Dots-and-Triangles is a first-player win when $n\neq 2$ \cite{Now}. We also prove that in both the open and closed $1\times n$ Dots-and-Boxes games where $n$ is even, the first player can guarantee a tie.\\
\end{abstract}

\section{The Game}

The classic children's game of Dots-and-Boxes has been well studied in \cite{Berl} and \cite{BCG}. The game begins with a square array of dots. Players take turns drawing an edge that connects two neighboring dots. A player who draws the fourth edge of a box claims the box and immediately takes another turn. The game ends when all boxes have been completed, and the player who has claimed more boxes is declared the winner. Variations of the game may be played, such as Dots-and-Triangles, which is played on a triangular board shape, and the Swedish and Icelandic games, in which some edges are drawn before the game begins \cite{Berl}.

In this paper we consider $1\times n$ (``narrow'') versions of both Dots-and-Boxes and Dots-and-Triangles. The $1\times n$ Dots-and-Boxes game consists of two rows of $n+1$ dots each; on completion of the game, $n$ boxes will have been enclosed. The $1\times n$ Dots-and-Triangles game consists of a row of $n$ dots on top and $n+1$ dots on bottom; on completion of the game, $2n-1$ triangles will have been enclosed. In the \emph{closed} narrow games, the top and side exterior edges have been drawn before the game begins. The \emph{open} version of the game begins with no such edges. The starting configurations for all four variants are shown in Figure~\ref{fig:gamestart}.\\
\indent To simplify our analysis, we will consider the graph-theoretic dual version of Dots-and-Boxes/Triangles known as Strings-and-Coins; this dual version of the game is thoroughly presented by Berlekamp in \cite{Berl}. In this dual game, each box or triangle in the original game corresponds to a vertex, or ``coin.'' Two vertices are adjacent in the dual game if and only if their corresponding faces share an edge \emph{that has not yet been drawn} in the original game. In addition, we think of the exterior of the game board as a single vertex called the ``ground.'' The ground vertex is not drawn. Drawing an edge in the original game corresponds to removing an edge, or ``cutting a string,'' in the dual game. Completing and claiming a face in the original game corresponds to isolating and taking the vertex, or ``capturing the coin,'' in the dual game. The ground vertex cannot be taken. Every face-enclosing game such as Dots-and-Boxes can be converted to an equivalent Strings-and-Coins game, although the converse is not true. The Strings-and-Coins presentations of the games in Figure~\ref{fig:gamestart} are shown in Figure~\ref{fig:stringsandcoins}. Following Berlekamp's convention, we use small arrows to denote an edge that goes to the ground. Note that the Strings-and-Coins presentation illustrates that closed Dots-and-Boxes and open Dots-and-Triangles are very similar; the latter is just the former with an extra edge at each extreme. 
\\
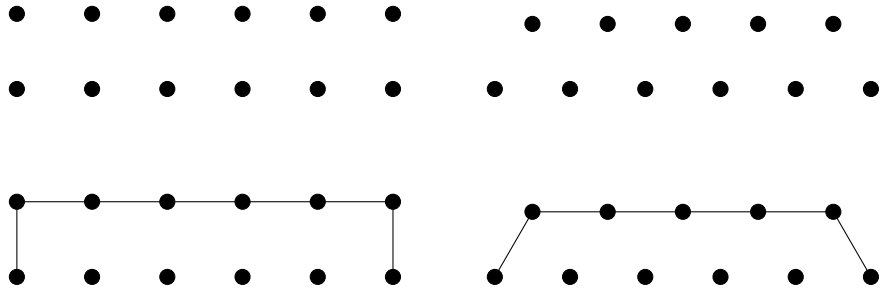
\begin{figure}[h]
  \centering
  \begin{tikzpicture}
    \foreach \x in {0,1,...,5} {
      \filldraw[dot] \DBlowdot{\x} circle (\nodesize);
      \filldraw[dot] \DBhighdot{\x} circle (\nodesize);
    }
    \begin{scope}[yshift=-2.5cm]
      \draw \DBlowdot0--\DBhighdot0;
      \foreach \x in {0,1,...,5} {
        \filldraw[dot] \DBlowdot{\x} circle (\nodesize);
        \filldraw[dot] \DBhighdot{\x} circle (\nodesize);
      }
      \foreach \x in {1,...,5} {
        \draw \DBhighdot{\x-1}--\DBhighdot{\x}; 
      }
      \draw \DBlowdot5--\DBhighdot5;
    \end{scope}
  \end{tikzpicture}
  \quad\quad\quad
  \begin{tikzpicture}
    \filldraw[dot] \DTlowdot0 circle (\nodesize);
    \foreach \x in {0,...,4} {
      \filldraw[dot] \DTlowdot{\x+1} circle (\nodesize);
      \filldraw[dot] \DThighdot{\x} circle (\nodesize);
    }
    \begin{scope}[yshift=-2.5cm]
      \draw \DTlowdot0--\DThighdot0;
      \filldraw[dot] \DTlowdot0 circle (\nodesize);
      \foreach \x in {0,...,4} {
        \filldraw[dot] \DTlowdot{\x+1} circle (\nodesize);
        \filldraw[dot] \DThighdot{\x} circle (\nodesize);
      }
      \foreach \x in {1,...,4} {
        \draw \DThighdot{\x-1}--\DThighdot{\x}; 
      }
      \draw \DThighdot4--\DTlowdot5;
    \end{scope}
  \end{tikzpicture}
  \caption{Starting positions for open (top) and closed (bottom) $1\times 5$ games of Dots-and-Boxes (left) and Dots-and-Triangles (right).}
	\label{fig:gamestart}

\end{figure}

\begin{figure}[h]
  \centering
  \begin{tikzpicture}
    \draw[->] \DBcoin0--\DBleftground0;
    \draw[->] \DBcoin4--\DBrightground4;
    \foreach \x in {1,...,4} {
      \draw \DBcoin{\x-1}--\DBcoin{\x};
    }
    \foreach \x in {0,...,4} {
      \draw[<->] \DBlowground{\x}--\DBcoin{\x}--\DBhighground{\x};
      \filldraw[coin] \DBcoin{\x} circle (\nodesize);
    }
    \begin{scope}[yshift=-2cm]
      \foreach \x in {1,...,4} {
        \draw \DBcoin{\x-1}--\DBcoin{\x};
      }
      \foreach \x in {0,...,4} {
        \draw[<->] \DBlowground{\x}--\DBcoin{\x};
        \filldraw[coin] \DBcoin{\x} circle (\nodesize);
      }
    \end{scope}
  \end{tikzpicture}
  \quad\quad
  \begin{tikzpicture}
    \draw[->] \DTlowcoin0--\DTleftground0;
    \draw[->] \DTlowcoin4--\DTrightground4;
    \foreach \x in {0,...,3} {
      \draw \DTlowcoin{\x}--\DThighcoin{\x}--\DTlowcoin{\x+1};
      \draw[->] \DThighcoin{\x}--\DThighground{\x};
      \filldraw[coin] \DThighcoin{\x} circle (\nodesize);
    }
    \foreach \x in {0,...,4} {
      \draw[<-] \DTlowground{\x}--\DTlowcoin{\x};
      \filldraw[coin] \DTlowcoin{\x} circle (\nodesize);
    }
    \begin{scope}[yshift=-2cm]
      \foreach \x in {0,...,3} {
        \draw \DTlowcoin{\x}--\DThighcoin{\x}--\DTlowcoin{\x+1};
        \filldraw[coin] \DThighcoin{\x} circle (\nodesize);
      }
      \foreach \x in {0,...,4} {
        \draw[<-] \DTlowground{\x}--\DTlowcoin{\x};
        \filldraw[coin] \DTlowcoin{\x} circle (\nodesize);
      }
    \end{scope}
  \end{tikzpicture}
  \caption{Strings-and-Coins equivalents to open (top) and closed (bottom) $1\times 5$ games of Dots-and-Boxes (left) and Dots-and-Triangles (right).}
	\label{fig:stringsandcoins}

\end{figure}
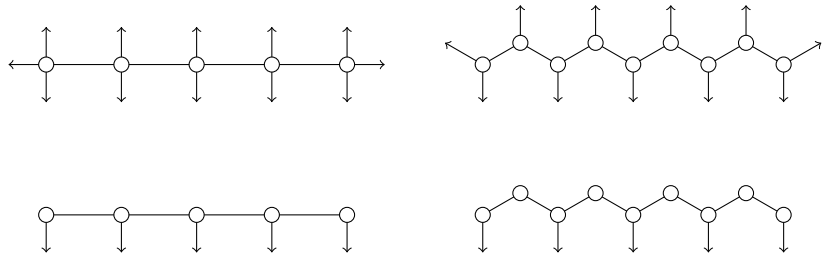

\section{Double-Dealing}
Astute players of Dots-and-Boxes are well acquainted with the strategy of double-dealing. We provide a brief review of this technique in this section as it plays an important role in the proof of our main result. Considering the Strings-and-Coins version of a closed narrow game, we say a vertex of degree one is an available vertex. We may assume that any player who has a winning strategy will take any available vertices, except in the special cases shown in Figure~\ref{fig:doubledealing}, as argued on pp. 42--43 of \cite{Berl}, where an available vertex is referred to as a ``capturable coin.'' In each of the three cases shown, a player has two options. Either the player may remove edge $X$ first and then edge $Y$, in which case they take two vertices and move again (if possible) in the remaining game, or the player may remove edge $Y$ first and end their turn. A player who is faced with any of the cases shown in Figure~\ref{fig:doubledealing} is said to have a \emph{double-dealing opportunity}. A player double-deals when he removes edge $Y$, thus declining to take two vertices that would have been taken had he removed edge $X$.\\
\begin{figure}[h]
  \centering
  \begin{tikzpicture}
    \draw[<-] \DTlowground0--\DTlowcoin0--\DThighcoin0;
    \draw \DTlowground0 coordinate (a);
    \draw \DTlowcoin0 coordinate (b);
    \draw ($(a)!0.5!(b)$) node[left] {Y};
    \draw \DThighcoin0 coordinate (a);
    \draw ($(a)!0.5!(b)$) node[above left] {X};
    \filldraw[coin] \DTlowcoin0 circle (\nodesize);
    \filldraw[coin] \DThighcoin0 circle (\nodesize);

    \draw[->] \DTlowcoin2--\DThighcoin2--\DTlowcoin3--\DTlowground3;
    \draw \DTlowcoin2 coordinate (a);
    \draw \DThighcoin2 coordinate (b);
    \draw ($(a)!0.5!(b)$) node[above left] {X};
    \draw \DTlowcoin3 coordinate (a);
    \draw ($(a)!0.95!(b)$) node[above right] {Y};
    \filldraw[markedcoin] \DTlowcoin3 circle (\nodesize);
    \filldraw[coin] \DThighcoin2 circle (\nodesize);
    \filldraw[coin] \DTlowcoin2 circle (\nodesize);

    \draw[->] \DTlowcoin5--\DThighcoin5--\DTlowcoin6--\DThighcoin6;
    \draw \DTlowcoin5 coordinate (a);
    \draw \DThighcoin5 coordinate (b);
    \draw ($(a)!0.5!(b)$) node[above left] {X};
    \draw \DTlowcoin6 coordinate (a);
    \draw ($(a)!0.75!(b)$) node[above right] {Y};
    \filldraw[coin] \DTlowcoin5 circle (\nodesize);
    \filldraw[coin] \DThighcoin5 circle (\nodesize);
    \filldraw[markedcoin] \DTlowcoin6 circle (\nodesize);
    \filldraw[markedcoin] \DThighcoin6 circle (\nodesize);

  \end{tikzpicture}
  \caption{Double-dealing opportunities. Each shaded vertex may be incident to other edges not shown in the figure.}
	\label{fig:doubledealing}
\end{figure}
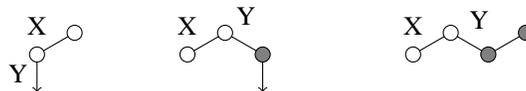

\indent A player who faces a double-dealing opportunity is in a powerful position, as explored in \cite{Berl, buzzard}. Let us consider a graph representing a position in a Dots-and-Boxes/Triangles game, and suppose that a player with a nonnegative net score\footnote{At any point in the game, Player A's score is the number of boxes/triangles which Player A has already claimed. In a game against Player B, Player A's net score is Player A's score minus Player B's score.} faces a double-dealing opportunity in the game. Let $G'$ denote the graph that will remain after edges $X$ and $Y$, as in Figure~\ref{fig:doubledealing}, have been removed and those two available vertices have been taken. Let $n$ denote the net advantage in score to the player who plays \emph{second} in $G'$, from that point until the game ends, assuming both players play optimally. Then the player with the double-dealing opportunity can choose whether or not to double deal. If they do not double-deal, then they take those two vertices and move first in $G'$. In this case, they will see their net score increase by $2-n$. If they double-deal, then they decline the two vertices. Their opponent can then collect them (or choose not to do so, though the opponent would derive no benefit from passing up these capture opportunities) and move first in $G'$, so that the double-dealing player will see their net score increase by at least $-2+n$. At least one of $2-n$ and $-2+n$ must be nonnegative. We may thus conclude as follows: \\
\begin{remark} The first player to be presented with a double-dealing opportunity at a point in the game when that player has a nonnegative net score can guarantee at least a tie. \end{remark}

\section{Base Graph and Chains}
Here we introduce some terminology specific to the Strings-and-Coins graphs representing positions in closed narrow Dots-and-Boxes/Triangles. The concept of a \emph{component} will be equivalent to the graph-theoretical conception of that notion, except that the components are not connected through the ground vertex. We say that an edge is a \emph{leg} if it is incident to the ground vertex; such incidences are conventionally represented by arrowheads. If a component has two or more legs, we say the two outermost legs are \emph{exterior legs}, and any other leg in that component is an \emph{interior leg}. 
\begin{definition} Given a graph $G$ representing a position in a game of narrow Dots-and-Boxes/Triangles, let $V_{BG}$ be the set of all vertices that are incident to an exterior leg, incident to an interior leg, or part of a path joining two vertices that are both incident to legs. Then the graph induced by $V_{BG}$ and the ground vertex is called the \emph{base graph}. Any vertices and edges that are not part of the base graph are said to be \emph{pendant vertices} and \emph{pendant edges}, respectively. 
\end{definition}

\begin{figure}[h]
  \centering
  \begin{tikzpicture}
    \draw[basegraph,<->] \DTlowground0--\DTlowcoin0--\DThighcoin0--
    \DTlowcoin1--\DThighcoin1--
    \DTlowcoin2--\DTlowground2;
    \filldraw[basegraph, coin] \DTlowcoin0 circle (\nodesize);
    \filldraw[basegraph, coin] \DThighcoin0 circle (\nodesize);
    \filldraw[basegraph, coin] \DTlowcoin1 circle (\nodesize);
    \filldraw[basegraph, coin] \DThighcoin1 circle (\nodesize);
    \filldraw[basegraph, coin] \DTlowcoin2 circle (\nodesize);

    \draw[<-] \DTlowground3--\DTlowcoin3;
    \filldraw[coin] \DTlowcoin3 circle (\nodesize);

    \draw \DTlowcoin4--\DThighcoin4--\DTlowcoin5;
    \filldraw[coin] \DTlowcoin4 circle (\nodesize);
    \filldraw[coin] \DThighcoin4 circle (\nodesize);
    \filldraw[coin] \DTlowcoin5 circle (\nodesize);

    \draw[basegraph,<->] \DTlowground6--\DTlowcoin6--\DThighcoin6--\DTlowcoin7--\DTlowground7;
    \draw \DTlowcoin7--\DThighcoin7;
    \filldraw[basegraph, coin] \DTlowcoin6 circle (\nodesize);
    \filldraw[basegraph, coin] \DThighcoin6 circle (\nodesize);
    \filldraw[basegraph, coin] \DTlowcoin7 circle (\nodesize);
    \filldraw[coin] \DThighcoin7 circle (\nodesize);
    
    \draw[<-] \DTlowground8--\DTlowcoin8--\DThighcoin8;
    \filldraw[coin] \DTlowcoin8 circle (\nodesize);
    \filldraw[coin] \DThighcoin8 circle (\nodesize);

    \draw[basegraph,<->] \DTlowground9--\DTlowcoin9--\DThighcoin9
    --\DTlowcoin{10}--\DThighcoin{10}--\DTlowcoin{11}--\DThighcoin{11}--\DTlowcoin{12}--\DTlowground{12};
    \draw \DTlowcoin{12}--\DThighcoin{12}--\DTlowcoin{13};
    \filldraw[basegraph, coin] \DTlowcoin9 circle (\nodesize);
    \filldraw[basegraph, coin] \DThighcoin9 circle (\nodesize);
    \filldraw[basegraph, coin] \DTlowcoin{10} circle (\nodesize);
    \filldraw[basegraph, coin] \DThighcoin{10} circle (\nodesize);
    \filldraw[basegraph, coin] \DTlowcoin{11} circle (\nodesize);
    \filldraw[basegraph, coin] \DThighcoin{11} circle (\nodesize);
    \filldraw[basegraph, coin] \DTlowcoin{12} circle (\nodesize);
    \filldraw[coin] \DThighcoin{12} circle (\nodesize);
    \filldraw[coin] \DTlowcoin{13} circle (\nodesize);

    \draw[->] \DThighcoin{13}--\DTlowcoin{14}--\DTlowground{14};
    \filldraw[coin] \DThighcoin{13} circle (\nodesize);
    \filldraw[coin] \DTlowcoin{14} circle (\nodesize);

  \end{tikzpicture}
  \caption{A game of Dots-and-Triangles in progress. The thick
    vertices and edges make up the base graph. All other vertices and edges are
    pendant vertices and edges.}
		\label{fig:basegraphs}
\end{figure}
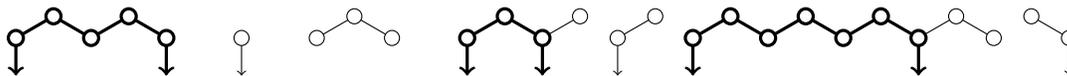

An example of the base graph of a game is depicted in Figure~\ref{fig:basegraphs}. Note that components with fewer than two legs do not have interior or exterior legs, so that they are definitionally excluded from the base graph. We partition the base graph into subgraphs called chains as follows.

\begin{definition}
A \emph{chain of length $k\geq0$} is a collection of $k+1$ edges in the base graph satisfying one of the following.
 \begin{itemize}

  \item Every interior leg is a chain of length zero.

  \item An exterior leg together with the edges of the path connecting it to the nearest interior leg in the same component comprise a chain, which excludes the interior leg.

  \item Edges of the path connecting two consecutive interior legs in the same component comprise a chain, again excluding the interior legs.

  \item If a component has two exterior legs and no interior legs, then all edges of the base graph in that component comprise a chain of length at least two.

  \end{itemize}

\end{definition}

Note that only edges of the base graph can be part of a chain. A player \emph{opens a chain} when the player removes one of its edges. Intuitively, we may think of the length of a chain as the number of vertices that can be taken by a player in a single turn once his opponent has opened the chain. Examples of chains of various lengths are shown in Figure~\ref{fig:chainexamples} for the Dots-and-Triangles game.

\begin{figure}
  \centering
  Dots-and-Triangles\\
  \begin{tikzpicture}
    \draw[firstchain,<-] \DTlowground0--\DTlowcoin0--\DThighcoin0--\DTlowcoin1;
    \draw \DTlowcoin0 node[above left,align=center] {chain\\(length 2)};

    \draw[secondchain,->] \DTlowcoin1--\DTlowground1;
    \draw \DTlowground1 node[below,align=center] {chain\\(length 0)};
    \draw[thirdchain] \DTlowcoin1--\DThighcoin1--\DTlowcoin2;
    \draw \DThighcoin1 node[above,align=center] {chain\\(length 1)};
    \draw[secondchain,->] \DTlowcoin2--\DTlowground2;
    \draw[firstchain,->] \DTlowcoin2--\DThighcoin2--\DTlowcoin3--\DThighcoin3--\DTlowcoin4--\DTlowground4;
    \draw \DThighcoin3 node[above,align=center] {chain\\(length 4)};
    \filldraw[chainother,coin] \DTlowcoin0 circle (\nodesize);
    \filldraw[chainother,coin] \DThighcoin0 circle (\nodesize);
    \filldraw[chainother,coin] \DTlowcoin1 circle (\nodesize);
    \filldraw[chainother,coin] \DThighcoin1 circle (\nodesize);
    \filldraw[chainother,coin] \DTlowcoin2 circle (\nodesize);
    \filldraw[chainother,coin] \DThighcoin2 circle (\nodesize);
    \filldraw[chainother,coin] \DTlowcoin3 circle (\nodesize);
    \filldraw[chainother,coin] \DThighcoin3 circle (\nodesize);
    \filldraw[chainother,coin] \DTlowcoin4 circle (\nodesize);

    \draw[<->,chainother] \DTlowground6--\DTlowcoin6--\DThighcoin6--\DTlowcoin7--\DTlowground7;
    \filldraw[chainother,coin] \DTlowcoin6 circle (\nodesize);
    \filldraw[chainother,coin] \DThighcoin6 circle (\nodesize);
    \filldraw[chainother,coin] \DTlowcoin7 circle (\nodesize);
    \draw \DTlowground{6.5} node[below,align=center] {chain\\(length 3)};

    \draw[<->,chainother] \DTlowground9--\DTlowcoin9--\DThighcoin9--\DTlowcoin{10}--\DThighcoin{10}--\DTlowcoin{11}--\DTlowground{11};
    \filldraw[chainother,coin] \DTlowcoin9 circle (\nodesize);
    \filldraw[chainother,coin] \DThighcoin9 circle (\nodesize);
    \filldraw[chainother,coin] \DTlowcoin{10} circle (\nodesize);
    \filldraw[chainother,coin] \DThighcoin{10} circle (\nodesize);
    \filldraw[chainother,coin] \DTlowcoin{11} circle (\nodesize);
    \draw \DTlowground{10} node[below,align=center] {chain\\(length 5)};
  \end{tikzpicture}
  \caption{Examples of Chains}
	\label{fig:chainexamples}
\end{figure}
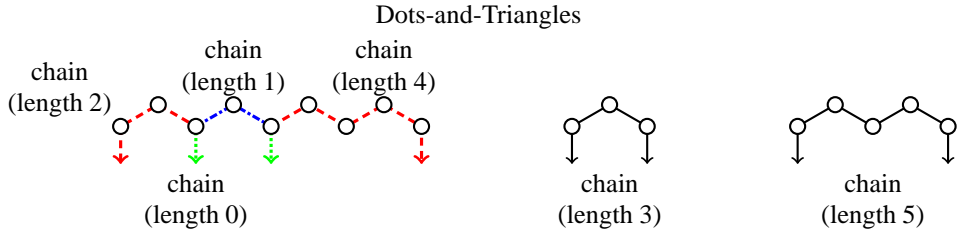

We say a chain is \emph{long} if its length is at least three, \emph{medium} if its length is exactly two, and \emph{short} if its length is zero or one. Observe that in Dots-and-Triangles, a chain has even length if and only if it contains exactly one exterior leg, although that is not the case with Dots-and-Boxes. It follows that, if a component in Dots-and-Triangles has two exterior legs and no interior leg, then the edges of that component make up a long chain. It is easily verified that any game of closed narrow Dots-and-Triangles eventually reduces to one or more components with two exterior legs and no interior legs. Thus, in a game of closed narrow Dots-and-Triangles, eventually one player must open a long chain.

We will call an edge \emph{bad} if its removal creates a double-dealing opportunity and \emph{good} otherwise. Based on the scenarios presented in Figure~\ref{fig:doubledealing}, we may characterize edges of the base graph as good or bad based on the length of their chains. Every edge of a long chain is bad. Unless it is incident to a pendant edge, every edge of a length-one chain is good. Chains of length zero are always good edges. In particular, in Dots-and-Triangles, every edge of a short chain is good. In a medium chain, the middle edge is good unless it is incident to a pendant edge, and the two outer edges are bad. The good and bad edges of some example short and medium chains are shown in Figure~\ref{fig:goodandbad}. In summary, 
\begin{proposition} An edge is good if and only if it is not incident to a pendant edge and either it is in a short chain or it is the middle edge in a medium chain.
\end{proposition}

\begin{figure}[h]

  \centering

  Short chains\\

  \begin{tikzpicture}

    \draw \DTlowcoin0--\DThighcoin0--\DTlowcoin1;

    \draw[very thick] \DTlowcoin1--\DThighcoin1--\DTlowcoin2;

    \draw \DTlowcoin2--\DThighcoin2--\DTlowcoin3;

    \draw \DTlowcoin1 coordinate (a);

    \draw \DThighcoin1 coordinate (b);

    \draw[stealth-,bend left=20,densely dotted] ($(a)!0.5!(b)$) to ++(-0.25,0.5) node[above] {good};

    \draw \DTlowcoin2 coordinate (a);

    \draw[stealth-,bend right=20,densely dotted] ($(a)!0.5!(b)$) to ++(0.25,0.5) node[above] {good};

    \foreach \x in {0,...,3} {

      \draw[<-] \DTlowground{\x}--\DTlowcoin{\x};

      \filldraw[coin] \DTlowcoin{\x} circle (\nodesize);

    }

    \foreach \x in {0,...,2} {

      \filldraw[coin] \DThighcoin{\x} circle (\nodesize);

    }

  \end{tikzpicture}\quad\quad
\begin{tikzpicture}

    \draw \DBcoin0--\DBcoin1;

    \draw[<-,very thick] \DBlowground1--\DBcoin1--\DBcoin2;

    \draw[<->] \DBlowground2--\DBcoin2--\DBcoin3--\DBlowground3;

    \draw[stealth-,bend left=20,densely dotted] \DBcoin{1.5} to ++(0.25,0.5) node[above] {bad};

    \draw \DBcoin1 coordinate (a);

    \draw \DBlowground1 coordinate (b);

    \draw[stealth-,bend right=20,densely dotted] ($(a)!0.5!(b)$) to ++(-0.5,-0.25) node[left] {bad};

    \foreach \x in {0,...,3} {

      \filldraw[coin] \DBcoin{\x} circle (\nodesize);

    }

  \end{tikzpicture}\\
	
	\vspace*{0.2in}

  Medium chains\\

  \begin{tikzpicture}

    \draw[<-,very thick] \DTlowground0--\DTlowcoin0--\DThighcoin0--\DTlowcoin1;

    \draw[->] \DTlowcoin1--\DThighcoin1--\DTlowcoin2--\DTlowground2;

    \draw[->] \DTlowcoin1--\DTlowground1;

    \draw \DTlowground0 coordinate (b);

    \draw \DTlowcoin0 coordinate (a);

    \draw[stealth-,bend right=20,densely dotted] ($(a)!0.5!(b)$) to ++(-0.5,-0.25) node[left] {bad};

    \draw \DThighcoin0 coordinate (b);

    \draw[stealth-,bend left=20,densely dotted] ($(a)!0.5!(b)$) to ++(-0.25,0.5) node[above] {good};

    \draw \DTlowcoin1 coordinate (a);

    \draw[stealth-,bend right=20,densely dotted] ($(a)!0.5!(b)$) to ++(0.25,0.5) node[above] {bad};

    \foreach \x in {0,...,2} {

      \filldraw[coin] \DTlowcoin{\x} circle (\nodesize);

    }

    \foreach \x in {0,...,1} {

      \filldraw[coin] \DThighcoin{\x} circle (\nodesize);

    }

  \end{tikzpicture}\quad\quad
 \begin{tikzpicture}

    \draw \DThighcoin{-1}--\DTlowcoin0;

    \draw[<-,very thick] \DTlowground0--\DTlowcoin0--\DThighcoin0--\DTlowcoin1;

    \draw[->] \DTlowcoin1--\DThighcoin1--\DTlowcoin2--\DTlowground2;

    \draw[->] \DTlowcoin1--\DTlowground1;

    \draw \DTlowground0 coordinate (b);

    \draw \DTlowcoin0 coordinate (a);

    \draw[stealth-,bend right=20,densely dotted] ($(a)!0.5!(b)$) to ++(-0.5,-0.25) node[left] {bad};

    \draw \DThighcoin0 coordinate (b);

    \draw[stealth-,bend left=20,densely dotted] ($(a)!0.5!(b)$) to ++(-0.25,0.5) node[above] {bad};

    \draw \DTlowcoin1 coordinate (a);

    \draw[stealth-,bend right=20,densely dotted] ($(a)!0.5!(b)$) to ++(0.25,0.5) node[above] {bad};

    \foreach \x in {0,...,2} {

      \filldraw[coin] \DTlowcoin{\x} circle (\nodesize);

    }

    \foreach \x in {-1,...,1} {

      \filldraw[coin] \DThighcoin{\x} circle (\nodesize);

    }

  \end{tikzpicture}

  \caption{Good and bad edges in short and medium chains.}

  \label{fig:goodandbad}

\end{figure}
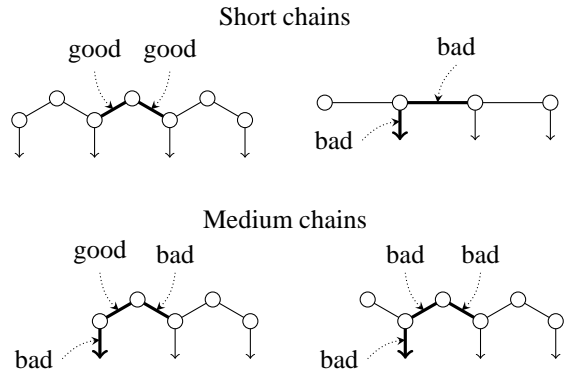

Observe that if a player opens a long chain of length $k\geq 3$, then $k$ vertices are made available to the opponent, two of which can be used to double-deal.

\section{Proof of Main Result}

The strategy we describe below is based on the notion of ``mirroring,'' whereby in a game with a symmetric scenario, one player can choose to mimic the other's moves, making use of the inherent symmetry of the game. In many games mirroring can be a means to force a tie, or to maintain an advantage, since both players perform identical actions and presumably obtain identical advantages. The dots games and their coin analogues do not have the necessary symmetries to make complete mirroring profitable, since on a turn in which a player captures a coin they would also capture its opposite, preserving the symmetry of the graph but not of the underlying game, wherein one player now has a score advantage.

A limited mirroring-like approach can, however, serve to preserve certain properties of the game, and the strategy described here makes use of such a ``quasi-mirroring'' approach. For brevity and ease of reference, we refer to the first player in this strategy description as ``Alice,'' and her opponent as ``Bob.'' We will show that Alice can create a roughly symmetric base graph on her first turn, and that by appropriately copying Bob's moves, she can end each turn by restoring the base graph to symmetry. She can follow this strategy until Bob is finally forced to open a long chain. This will not be traditional mirroring, as Alice's moves are not always an exact mirror of Bob's. We describe this strategy below to prove our main result:

\begin{theorem}\label{thm:main} A game of closed $1\times n$ Dots-and-Triangles where $n\neq 2$ is a first-player win.\end{theorem}

\begin{proof} The $n=1$ case is trivial, so we suppose that $n\geq 3$. On Alice's first turn, she can take the edge closest to the middle of the graph. If $n$ is odd, this will be the center interior leg. If $n$ is even, this will be one of the center non-leg edges. Figure~\ref{fig:firstturn} shows examples of both parities; note that when $n$ is odd the graph at the end of Alice's first turn will be symmetric. 

\begin{figure}[h]

  \centering

  \begin{tabular}{lc}

    $n=4$&

    \begin{tikzpicture}[baseline=(current bounding box.center)]

      \foreach \x in {0,...,2} {

        \draw[<-] \DTlowground{\x}--\DTlowcoin{\x}--\DThighcoin{\x};

      }

      \draw \DThighcoin0--\DTlowcoin1;

      \draw[->] \DThighcoin2--\DTlowcoin3--\DTlowground3;

      \foreach \x in {0,...,3} {

        \filldraw[coin] \DTlowcoin{\x} circle (\nodesize);

      }      

      \foreach \x in {0,...,2} {

        \filldraw[coin] \DThighcoin{\x} circle (\nodesize);

      }

    \end{tikzpicture}\\

    &\\

    $n=5$&

    \begin{tikzpicture}[baseline=(current bounding box.center)]

      \foreach \x in {0,1,3} {

        \draw[<-] \DTlowground{\x}--\DTlowcoin{\x}--\DThighcoin{\x}--\DTlowcoin{\x+1};

      }

      \draw \DTlowcoin2--\DThighcoin2--\DTlowcoin3;

      \draw[->] \DTlowcoin4--\DTlowground4;

      \foreach \x in {0,...,4} {

        \filldraw[coin] \DTlowcoin{\x} circle (\nodesize);

      }      

      \foreach \x in {0,...,3} {

        \filldraw[coin] \DThighcoin{\x} circle (\nodesize);

      }

    \end{tikzpicture}\\

  \end{tabular}

  \caption{The game at the end of Alice's first turn}

  \label{fig:firstturn}

\end{figure}
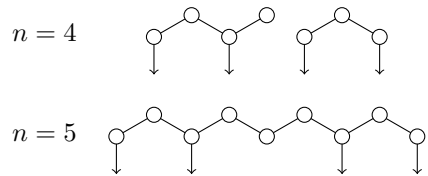
If $n$ is even, then at the beginning of Bob's first turn he faces a symmetric base graph with exactly one available pendant vertex, which he may or may not claim, while if $n$ is odd then the graph is symmetric with no pendant edges. Whether $n$ is even or odd, Bob must end this turn by removing an edge in the symmetric base graph. At the end of Bob's first turn, Alice's net score will thus be either $0$ or $-1$.

Until a long chain is opened, Alice adopts the following quasi-mirroring strategy.
\begin{itemize}

\item If Bob takes a good edge, Alice takes any available vertices and mirrors Bob's move; i.e., she takes the mirror image of the edge removed by Bob.

\item If Bob takes the middle edge of a medium chain, and that edge is incident to a pendant, then Alice takes any available vertices and mirrors Bob's move.

\item If Bob takes one of the outer edges of a medium chain, Alice takes any available vertices and takes the good edge in the chain's mirror image.

\end{itemize}

Assuming that such moves are always possible for Alice, it is clear that Bob will be forced to be the first to open a long chain, and that until then, Alice will not give him a double-dealing opportunity because she only takes good edges. We now argue that if Bob takes an edge that is not part of a long chain (as in the three possible moves described above), and if that edge is in a symmetric base graph, then Alice's response moves detailed above are possible and will restore the base graph to symmetry at the end of her turn.

We shall presume by way of induction that on Bob's turn, he is faced with a symmetric base graph with no double-dealing opportunity; note that this is certainly the case on his first turn. After taking any number of pendant vertices, he must make a move in the base graph. If he takes an edge from a long chain, then symmetry becomes irrelevant because Alice's strategy changes as soon as Bob opens a long chain. Suppose contrariwise that he takes an edge that is not in a long chain. Call this edge $e$ and its mirror image $e'$. We observe that $e'$ must still be in the base graph at the end of Bob's turn: if it were not, then $e'$ would have to be in the same chain as $e$. However, the two possible graph layouts appearing in Figure~\ref{fig:firstturn} ensure that a vertex and its mirror image must be either in different components or separated by a long chain. In further analyzing Alice's moves, we must consider the distinct cases when edge $e$ is good or bad.

If $e$ is a good edge, then since the base graph is symmetric, $e'$ is either good or incident to a pendant. If $e'$ is incident to a pendant, then Alice will remove the pendant before taking $e'$, and then $e'$ will be good when Alice takes it.

Now suppose $e$ is a bad edge. Then $e$ must be in a medium chain, since every edge of a short chain is good in Dots-and-Triangles. Thus, either $e$ is the middle edge of the chain and incident to a pendant edge, or $e$ is one of the two outer edges of the chain; see Figure~\ref{fig:goodandbad}. Suppose $e$ is the middle edge of a medium chain and $e$ is incident to a pendant edge. Since Alice takes all available vertices before making a move in the base graph, the only pendant edges/vertices at the start of any of Bob's turns must have been created by Alice's most recent move, and so there can only be pendants on one side of the graph. Therefore, if $e$ is incident to a pendant edge, $e'$ is not. So $e'$ is a good edge when Alice takes it. Now suppose $e$ is one of the outer edges in a medium chain. Then $e'$ is also bad. However, by the strategy's established response to Bob removing an outer edge of a medium chain above, Alice does not take $e'$; instead she takes the good edge in the chain's mirror image. Since this move gets rid of the chain containing $e'$, it still restores the base graph to symmetry before Bob's next turn. See Figure~\ref{fig:mirrorbreak}.

\begin{figure}[h]

  \centering

  \begin{tikzpicture}

    \foreach \x in {0,1,3,4} {

      \draw \DTlowcoin{\x}--\DThighcoin{\x}--\DTlowcoin{\x+1};

    }

    \foreach \x in {0,...,5} {

      \draw[->] \DTlowcoin{\x}--\DTlowground{\x};

    }

    \foreach \x in {0,...,5} {

      \draw[->] \DTlowcoin{\x}--\DTlowground{\x};

      \filldraw[coin] \DTlowcoin{\x} circle (\nodesize);

    }      

    \foreach \x in {0,1,3,4} {

      \filldraw[coin] \DThighcoin{\x} circle (\nodesize);

    }

    \draw \DTlowcoin0 node[above left] {$v_1$};

    \draw \DThighcoin0 node[above] {$v_2$};

    \draw \DTlowcoin0 coordinate (a);

    \draw \DTlowground0 coordinate (b);

    \draw ($(a)!0.5!(b)$) node[left] {$e$};

    \draw \DTlowcoin5 coordinate (a);

    \draw \DTlowground5 coordinate (b);

    \draw ($(a)!0.5!(b)$) node[left] {$e'$};

    \draw \DTlowcoin5 coordinate (a);

    \draw \DThighcoin4 coordinate (b);

    \draw ($(a)!0.5!(b)$) node[above] {$f$};

  \end{tikzpicture}

  \caption{Bob removes edge $e$. Alice collects vertices $v_1$ and $v_2$, and ends her turn by removing good edge $f$ rather than (bad) mirror edge $e'$.}

  \label{fig:mirrorbreak}

\end{figure}
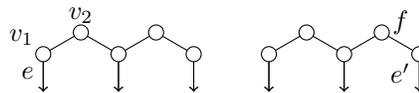

Note that one player must eventually open a long chain. By following the strategy above, Alice can force Bob to be the first one to do so. Thus Bob will open a long chain before he faces a double-dealing opportunity, as Alice only removes good edges. At the beginning of her second turn, Alice has a net score of at least $-1$. By following the quasi-mirroring strategy and taking any available vertices, Alice never frees up more vertices than she takes on any given turn. Therefore, at the start of \emph{any} of her turns, her net score is at least $-1$. When Bob finally opens a long chain of length $k\geq 3$, Alice may take $k-2\geq 1$ of the available vertices, bringing her net score up to at least zero. She then has a double-dealing opportunity. By Remark 1, she can guarantee at least a tie. But the number of triangles is odd, so there are no ties. Thus Alice wins.
\end{proof}

\section{Related Results}
A similar strategy can be employed in the narrow Dots-and-Boxes game to prove the following.
\begin{theorem}\label{thm:boxes} In a game of closed or open $1\times n$ Dots-and-Boxes where $n\geq 4$ is even, the first player can guarantee a tie.\end{theorem}
\begin{proof} As before, let Alice and Bob be players 1 and 2, respectively. Again, Alice's first move it to take the center edge; in this case it will be a non-leg edge; see Figure~\ref{fig:firstmoveboxes}. Alice follows the quasi-mirroring strategy as outlined in the proof above. The double and triple edges to the ground at a single vertex in the open game do not affect her strategy. In fact, one may view each double (respectively, triple) edge to the ground as a single edge with weight $2$ ($3$), where each edge removal is equivalent to reducing the weight by $1$. At the end of her second turn, Alice's net score is nonnegative; by the same argument as above, her net score will be nonnegative at the end of any of her turns. It is possible for Bob to play in such a way that he is never forced to open a long chain. However, it is still true that Bob is the first player to give his opponent a double-dealing opportunity, if this occurs. By copying moves appropriately and denying Bob a double-dealing opportunity, Alice can guarantee a net score of at least zero when the game ends.

\begin{figure}[h]

  \centering

  \begin{tabular}{lc}

    closed&

    \begin{tikzpicture}[baseline=(current bounding box.center)]

      \foreach \x in {1,2,4,5} {

        \draw \DBcoin{\x-1}--\DBcoin{\x};

      }

      \foreach \x in {0,...,5} {

        \draw[<-] \DBlowground{\x}--\DBcoin{\x};

        \filldraw[coin] \DBcoin{\x} circle (\nodesize);      

      }

    \end{tikzpicture}

    \\

    &\\

    open&

    \begin{tikzpicture}[baseline=(current bounding box.center)]

      \foreach \x in {1,2,4,5} {

        \draw \DBcoin{\x-1}--\DBcoin{\x};

      }

      \foreach \x in {0,5} {

        \draw[<-] \DBlowground{\x} ++({2*\nodesize},0)--\DBcoin{\x};

        \draw[<-] \DBlowground{\x}--\DBcoin{\x};

        \draw[<-] \DBlowground{\x} ++({-2*\nodesize},0)--\DBcoin{\x};

        \filldraw[coin] \DBcoin{\x} circle (\nodesize);      

      }

      \foreach \x in {1,2,3,4} {

        \draw[<-] \DBlowground{\x} ++({\nodesize},0)--\DBcoin{\x};

        \draw[<-] \DBlowground{\x} ++({-1*\nodesize},0)--\DBcoin{\x};

        \filldraw[coin] \DBcoin{\x} circle (\nodesize);      

      }

    \end{tikzpicture}

    \\

  \end{tabular}

  \caption{The game after Alice has removed the first edge in the $n=6$ case.}

  \label{fig:firstmoveboxes}

\end{figure}
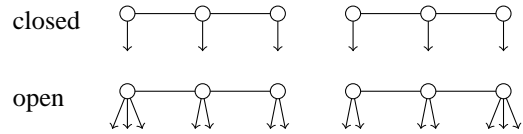
\end{proof}

One may wonder why we stipulate that $n$ must be even in the Dots-and-Boxes case. The reason is that when $n$ is odd and Alice attempts to use the quasi-mirroring strategy, Bob has a move that is ``unmirrorable;'' see Figure~\ref{fig:oddnomirror}. In this case both $e$, the edge to be removed by Bob, and $e'$, the edge to be removed by Alice, are part of the same chain before Bob removes $e$. If Alice removes $e'$, then she must move again in the base graph and now Bob may employ the quasi-mirroring strategy. In fact, if Alice attempts to use mirroring until Bob takes an unmirrorable edge, then Bob can win by taking interior legs in sequence until the game is in the position shown in Figure~\ref{fig:oddloss}, at which point Alice must open a long chain and Bob can win the game. It seems that something more complicated than a quasi-mirroring strategy is called for in the odd versions of narrow Dots-and-Boxes. 

\section*{$\;$}

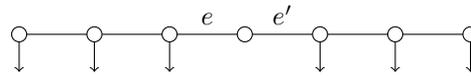
\begin{figure}[h]

  \centering

  \begin{tikzpicture}
    \foreach \x in {1,...,6} {

      \draw \DBcoin{\x-1}--\DBcoin{\x};

    }

    \foreach \x in {0,1,2,4,5,6} {

      \draw[<-] \DBlowground{\x}--\DBcoin{\x};

    }

    \foreach \x in {0,...,6} {

      \filldraw[coin] \DBcoin{\x} circle (\nodesize);      

    }

    \draw \DBcoin{2.5} node[above] {$e$};

    \draw \DBcoin{3.5} node[above] {$e'$};

  \end{tikzpicture}

  \caption{Alice has removed the center leg on her first turn. If Bob removes edge $e$ and Alice responds by removing its mirror $e'$, then Alice must move again before ending her turn.}

  \label{fig:oddnomirror}

\end{figure}
\section*{$\;$}

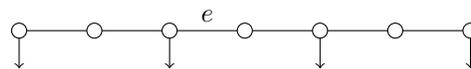
\begin{figure}[h]

  \centering

  \begin{tikzpicture}

    \foreach \x in {1,...,6} {

      \draw \DBcoin{\x-1}--\DBcoin{\x};

    }

    \foreach \x in {0,2,4,6} {

      \draw[<-] \DBlowground{\x}--\DBcoin{\x};

    }

    \foreach \x in {0,...,6} {

      \filldraw[coin] \DBcoin{\x} circle (\nodesize);      

    }

    \draw \DBcoin{2.5} node[above] {$e$};

  \end{tikzpicture}

  \caption{If Bob takes edge $e$, Alice will be forced to open a long chain.}

  \label{fig:oddloss}

\end{figure}
\section{Further Work}

There are several open questions regarding the $1\times n$ Dots-and-Boxes/Triangles games. Of immediate interest are the narrow Dots-and-Boxes game where $n$ is odd, and the open narrow Dots-and-Triangles game, as our results do not apply to these cases. Since their graph-theoretic duals are so similar (see Figure~\ref{fig:stringsandcoins}), it seems that a winning strategy for one would be a winning strategy for the other.

Through a computer search, we have considered the closed $1\times n$ Dots-and-Boxes game up to $n=21$. The results are shown in Table~\ref{tab:netscore}. Our results suggest that under optimal play (i.e., playing so as to maximize final net score), the first player can win with a final net score of 1 when $n$ is odd and $n\geq 9$, and that the first player can end with a final net score of 0 when $n$ is even and $n\geq 4$. This suggests that Theorem~\ref{thm:boxes} cannot be improved upon, and that the closed $1\times n$ Dots-and-Boxes game is a first-player win when $n$ is odd. We have also considered games of closed $1\times n$ Dots-and-Triangles up to $n=15$. The quasi-mirroring strategy employed in the proof above guarantees a win but does not necessarily maximize Player 1's final net score. Of particular interest are games where $n$ is a multiple of 3. It seems that Player 1 can win with a final net score of 5 when $n=3,6,9,12$, but this pattern does not continue, as the final net score is just $1$ in the $n=15$ case. A strategy that is known to maximize the final net score would be of interest here.

\begin{table}[h]

\centering

\begin{tabular}{c c| c c}\hline

\multicolumn{2}{c}{Dots-and-Boxes} & \multicolumn{2}{c}{Dots-and-Triangles}\\ \hline

$n$ & score & $n$ & score\\ \hline

1 & 1 & 1 & 1\\

2 & -2 & 2 & -3\\ 

3 & 3 & 3 & 5\\

4 & 0 & 4 & 1\\

5 &1 & 5  & 1\\

6 &0 & 6 & 5\\

7 &3 & 7 & 3\\

8 &0 & 8 & 1\\

9 &1 & 9 &  5\\

10 &0 & 10 & 1\\

11 &1 & 11 & 1\\

12 &0 & 12 & 5\\

13 &1 & 13 &  1\\

14 &0 & 14 & 1\\

15 & 1 & 15 & 1\\

16 &0 & \; & \;\\

17 &1 & \; & \;\\

18 &0 & \; & \;\\

19 &1 & \; & \;\\

20 &0 & \; & \;\\

21 &1  & \; & \;\\ \hline

\end{tabular}

\caption{Maximum final net score for Player 1 in closed versions of the games.}

\label{tab:netscore}

\end{table}

\end{document}